\documentclass{amsart}
\usepackage{amsmath}
\usepackage{amsfonts}
\usepackage{amssymb}
\usepackage{graphicx}
\usepackage{amsthm}
\usepackage{cite}
\usepackage{tikz}

\theoremstyle{plain}
\newtheorem{thm}{Theorem}
\newtheorem*{nthm}{Theorem}
\newtheorem{lem}{Lemma}
\newtheorem{prop}{Proposition}

\theoremstyle{definition}

\newtheorem*{ques}{Question}

\theoremstyle{remark}

\begin{document}
\title[Non-locally (OB) example]{An automorphism group of an $ \omega $-stable structure that is not locally (OB)}
\author{Joseph Zielinski}
\address{Department of Mathematics, Statistics, and Computer Science,
University of Illinois at Chicago,
322 Science and Engineering Offices (M/C 249),
851 S. Morgan Street,
Chicago, IL 60607-7045}
\email{zielinski.math@gmail.com}
\urladdr{http://math.uic.edu/~zielinski}
\begin{abstract}
We observe that there is an example of an automorphism group of a model of an $ \omega $-stable theory---in fact, the prime model of an uncountably-categorical theory---that is not locally (OB), answering a question of C. Rosendal.
\end{abstract}
\maketitle

\section{Introduction}
Towards extending the techniques of geometric group theory to all topological groups, C. Rosendal, in \cite{2014arXiv1403.3106R}, identifies, for a general topological group, the appropriate notion of ``boundedness''.
The sets with this property play the role of the compact subsets of a locally-compact group and norm-bounded subsets of (the additive group of) a Banach space---and indeed, coincide with these examples for the above classes of groups.
Here, the sets with the \emph{relative property (OB)} are those that are inexorably bounded, in the sense that they take finite diameter with respect to \emph{every} continuous, left-invariant pseudometric on the group.

Recall that a \emph{coarse structure} on a set, $ X $, is any family of subsets of $ X^{2} $ extending the powerset of the diagonal and closed under subsets, unions, inverses, and compositions of relations. For example, a coarse structure naturally arising from a metric space, $ (X,d) $, consists of those sets $ E \subseteq X^{2} $ such that $ \sup \{d(x,y) \mid (x,y) \in E\} $ is finite (see \cite{MR2007488}).
The family of sets with the relative property (OB) forms an ideal, stable under the group operations, and thereby induces a left-invariant coarse structure on the group generated by the entourages, $ \{(x,y) \mid x^{-1}y \in A  \} $, as $ A $ varies over relatively (OB) sets.

Associated to this concept are several attributes that a given topological group may possess. A group is \emph{locally (OB)} if there is an open neighborhood of the identity element with the property (OB). For a broad class of topological groups, this completely coincides with the situation where the above coarse structure may be given by a metric. Additionally, the group simply \emph{has the property (OB)} when every subset has the relative property (OB) as above, i.e., when the group has finite diameter with respect to every continuous, left-invariant pseudometric. These are the groups for which the above coarse structure is trivial.

Much of the motivation for better understanding ``large'' topological groups is that they often arise as transformations of important mathematical objects, e.g., homeomorphism groups of compact topological spaces, isometry groups of metric spaces, diffeomorphism groups of manifolds, and automorphism groups of countable structures in model theory. The coarse geometry of the groups from this latter class received a more thorough treatment in \cite{2014arXiv1403.3107R}. One of the main results of that paper was,

\begin{nthm}[Rosendal]
If $ \mathcal{M} $ is the countable, saturated model of an $ \omega $-stable theory, then $ \operatorname{Aut}(\mathcal{M}) $ has the property (OB).
\end{nthm}
Recall that a theory is \emph{$ \omega $-stable} when there are only countably-many complete types over countable parameter sets, and that a countable model is \emph{saturated} when it realizes all types over finite parameter sets. This theorem and its proof led to the conjecture,

\begin{ques}[Rosendal]
If $ \mathcal{M} $ is any model of an $ \omega $-stable theory, must $ \operatorname{Aut}(\mathcal{M}) $ be locally (OB)?
\end{ques}

Here we answer this question is the negative, namely we show,
\begin{thm} \label{counterexample}
There is a countable structure, $ \mathcal{M} $, for which $ \operatorname{Th}(\mathcal{M}) $ is uncountably-categorical, $ \mathcal{M} $ is its prime model, and $ \operatorname{Aut}(\mathcal{M}) $ is not locally (OB).
\end{thm}

\section{Preliminaries}

The important notions for the coarse geometry of automorphism groups were introduced above. Let us recall here, from \cite{2014arXiv1403.3106R}, an important source of examples of groups that are not locally (OB).

\begin{lem} \label{products}
A product of groups, $ G = \prod_{i \in I} G_{i} $, is locally (OB) if and only if $ G_{i} $ has the (full) property (OB) for all-but-finitely-many $ i \in I $.
\end{lem}

\begin{proof}
As it suffices for our needs here, we will show just (the contrapositive of) the ``only if'' direction, and for metrizable $ G $. A full proof is found in Proposition 13 of \cite{2014arXiv1403.3106R}. Let $ U \subseteq G $ be open. Then for some $ j \in I $, the projection of $ U $ onto the $ j $th coordinate is all of $ G_{j} $, and $ G_{j} $ does not have the property (OB). Let $ d $ be a compatible, left-invariant metric on $ G_{j} $ of infinite diameter, and let $ \rho $ be any compatible, left-invariant metric on $ G $. Then $ \rho((g_{i})_{i \in I},(h_{i})_{i \in I}) + d(g_{j},h_{j}) $ is a compatible, left-invariant metric on $ G $ assigning infinite diameter to $ U $. Therefore, $ U $ does not have the property (OB), and as it was an arbitrary open set, $ G $ is not locally (OB).
\end{proof}

For our purposes, ``theory'' will mean ``full theory of an infinite structure in a countable language''. Recall that, for an infinite cardinal, $ \kappa $, a theory is \emph{$ \kappa $-categorical} if it has exactly one model, up to isomorphism, of cardinality $ \kappa $, and that by a foundational result of M. Morley, a theory is categorical in one uncountable cardinal if and only if it is categorical in all uncountable cardinals. Such theories are then unambiguously termed \emph{uncountably categorical}. If such a theory is also $ \omega $-categorical, then it is said to be \emph{totally categorical}.

Given a structure, $ \mathcal{M} $, and a tuple $ \bar{a} \in M^{n} $, a formula $ \varphi(x,\bar{a}) $ is \emph{strongly minimal} if it defines an infinite set, and in every elementary extension of $ \mathcal{M} $, every further definable subset is either finite or cofinite. Strongly minimal formulas (and the \emph{strongly minimal sets} they define) play a fundamental role in the structure theory of uncountably-categorical theories.

\begin{nthm}[Baldwin-Lachlan \cite{JSL:9103821}]
If a theory, $ T $, is uncountably categorical, then there is a strongly-minimal $ \varphi(x,\bar{a}) $ (with parameters from the prime model), and models of $ T $ are determined, up to isomorphism, by the minimal cardinality of a set, $ B \subseteq \varphi(\mathcal{M},\bar{a}) $ for which $ \varphi(\mathcal{M},\bar{a}) \subseteq \operatorname{acl}(\bar{a} \cup B) $.
\end{nthm}

Here, $ \varphi(\mathcal{M},\bar{a}) $ is the set of points in $ M $ defined by $ \varphi(x,\bar{a}) $, and the \emph{algebraic closure} of a set, $ C \subset M $, is $ \operatorname{acl}(C) = \bigcup \{\varphi(\mathcal{M},\bar{c}) \mid \bar{c} \subseteq C \text{ and } | \varphi(\mathcal{M},\bar{c}) | < \infty \} $, the union of all finite sets definable with parameters from $ C $.
A theory is \emph{strongly minimal} if $ x = x $ is strongly minimal (i.e., every definable subset of every model is finite or cofinite), and \emph{almost strongly minimal} if every model is algebraic over a strongly minimal set. Such theories are uncountably categorical.

Let us remark that the example in the following section has a theory that is uncountably categorical, but not totally categorical, not almost strongly minimal, and for which $ \mathcal{M} $ is not the saturated model. In fact, for an uncountably categorical structure to be a counterexample, these additional properties are necessary.

\begin{prop}
Suppose $ \mathcal{M} $ is a countable structure and $ T = \operatorname{Th}(\mathcal{M}) $ is uncountably categorical. Then if $ T $ is totally categorical, or more generally if $ \mathcal{M} $ is the countable, saturated model of $ T $, or if $ T $ is almost strongly minimal, then $ \operatorname{Aut}(\mathcal{M}) $ is locally (OB).
\end{prop}

\begin{proof}
The case where $ \operatorname{Th}(\mathcal{M}) $ is $ \omega $-categorical is due to P. Cameron, and as mentioned above, was extended by Rosendal to saturated models of $ \omega $-stable theories. In both cases $ \operatorname{Aut}(\mathcal{M}) $ has the property (OB). So suppose that $ \mathcal{M} $ is not saturated, but that $ \operatorname{Th}(\mathcal{M}) $ is almost strongly minimal. Then there is a strongly minimal formula, $ \varphi(x,\bar{a}) $, and $ M = \operatorname{acl}(\varphi(\mathcal{M},\bar{a})) $. As $ \mathcal{M} $ is not saturated, there is a finite $ B \subseteq \varphi(\mathcal{M},\bar{a}) $ so that $ \varphi(\mathcal{M},\bar{a}) \subseteq \operatorname{acl}(\bar{a} \cup B) $.

Therefore, $ M = \operatorname{acl}(\bar{a} \cup B) $. Let $ V \leq \operatorname{Aut}(\mathcal{M}) $ be the stabilizer subgroup of $ \bar{a} \cup B $. Then $ V $ is an open subgroup and as $ \mathcal{M} $ is algebraic over $ \bar{a} \cup B $, every $ c \in M $ has a finite $ V $-orbit. Therefore (see \cite{Cameron01031996}), $ V $ is compact, and so $ \operatorname{Aut}(\mathcal{M}) $ is locally-compact, and thus locally (OB).
\end{proof}

We remark that all natural and frequently-cited examples of uncountably categorical structures possess at least one of the aforementioned properties, contributing to the plausibility of the conjecture refuted here.

\section{The example}
The example here, of an uncountably categorical structure whose automorphism group is not locally (OB), is essentially that of section 4 of \cite{JSL:9103821}, with some modifications. Therefore, in the proof of categoricity we will provide only an outline, referring the reader to the above paper for more details on that aspect.

Let $ \mathcal{L} = \{f,R,0,1 \} $, a language with a ternary function symbol, a binary relation symbol, and two constants. Let $ M = \mathbb{Q} \cup \mathbb{Q}^{2} $. Interpret the symbols $ 0 $ and $ 1 $ as the corresponding elements of $ \mathbb{Q} $, and let $ R^{\mathcal{M}}= \{(p,(p,q)) \in M^{2} \mid p,q \in \mathbb{Q} \} $.
We define $ f^{\mathcal{M}} $ by cases:
\begin{itemize}
\item $ f^{\mathcal{M}}(p,q,r) = (q-p) + r $ for $ p,q,r \in \mathbb{Q} $
\item $ f^{\mathcal{M}}((s,p),(s,q),(s,r)) = (s, (q-p)+r) $ for $ p,q,r,s \in \mathbb{Q} $
\item $ f^{\mathcal{M}}(p,(p,q),r) = (p,(q-p)+r) $ for $ p,q,r \in \mathbb{Q} $
\item $ f^{\mathcal{M}}(a,b,c) = c $ if $ (a,b,c) $ is not one of the above forms
\end{itemize}

One should keep in mind the following picture of $ \mathcal{M} $:  It consists of a ``parent'' copy of $ \mathbb{Q} $, and corresponding to each of its elements, a ``child'' copy of $ \mathbb{Q} $. The elements $ 0 $ and $ 1 $ of the parent copy are distinguished, and the relation $ R^{\mathcal{M}}(a,b) $ holds precisely when $ a $ is a member of the parent copy and $ b $ is a member of the child copy associated to $ a $.

The function $ f^{\mathcal{M}} $ is best considered not as a three-variable function, but as a family of single-variable functions parameterized by pairs of elements of $ \mathcal{M} $. That is, $ f^{\mathcal{M}}(a,b,c) $ should be viewed as the value that $ c $ takes in the function determined by $ (a,b) $. So the first condition says that if $ a $ and $ b $ are both in the parent copy of $ \mathbb{Q} $, then $ f^{\mathcal{M}}(a,b,\cdot) $ acts as a translation of the parent copy by $ (b-a) $, and as the identity on the child copies. Similarly if $ a $ and $ b $ are in the same child copy, then $ f^{\mathcal{M}}(a,b,\cdot) $ translates that child copy. The third case is probably the least intuitive, but if $ c $ is in the parent copy and $ R^{\mathcal{M}}(a,b) $, then $ f^{\mathcal{M}}(a,b,c) $ is best described as ``where $ c $ would go if the parent copy was laid on top of the child copy corresponding to $ a $, in such a way that $ a $ was made to line up with $ b $''.

\begin{prop} [Baldwin-Lachlan] \label{categorical}
$ \operatorname{Th}(\mathcal{M}) $ is uncountably categorical.
\end{prop}

\begin{proof}
First, note that it suffices to show that $ \mathcal{M}' $, the reduct of $ \mathcal{M} $ to the language $ \mathcal{L}' = \{f,R\} $, is uncountably categorical. Next, we see that the structure $ (\mathbb{Q},F) $ where $ F(p,q,r) = (q-p) + r $ is strongly minimal. To see this, first verify, by induction, that for every term $ t(x_{1},\dots,x_{n}) $ there are $ r_{1},\dots,r_{n} \in \mathbb{Z} $, summing to $ 1 $, so that $ (\mathbb{Q},F) $ interprets $ t(a_{1},\dots,a_{n}) $ as $ r_{1}a_{1}+ \cdots + r_{n}a_{n} $ for every choice of $ a_{1},\dots,a_{n} $. Therefore, for every atomic formula $ \varphi(x_{1},\dots,x_{n}) $ there are $ r_{1},\dots,r_{n} \in \mathbb{Z} $, summing to $ 0 $, so that $ (\mathbb{Q},F) \models \varphi(a_{1},\dots,a_{n}) $ if and only if $ r_{1}a_{1}+ \cdots + r_{n}a_{n} = 0 $. Then by induction on the construction of formulas, every $ \emptyset $-definable relation in $ (\mathbb{Q},F) $ is a Boolean combination of sets of the form
\[ \{(a_{1},\dots,a_{n}) \in \mathbb{Q}^{n} \mid r_{1}a_{1}+ \cdots + r_{n}a_{n} = 0 \} \]
where the $ r_{i} $'s sum to $ 0 $. So for every $ \varphi(x_{1},\dots,x_{n}) $, there is a $ k_{\varphi} \in \mathbb{N} $ so that for any $ a_{2},\dots,a_{n} $, either $ \varphi(x_{1},a_{2},\dots,a_{n}) $ or $ \neg \varphi(x_{1},a_{2},\dots,a_{n}) $ has at most $ k_{\varphi} $ solutions. This fact is expressible in a first-order manner, and so in every model of $ \operatorname{Th}(\mathbb{Q},F) $, every set defined by $ \varphi $ with parameters is either of size less than $ k_{\varphi} $ or has compliment with this bound. Therefore, $ \operatorname{Th}(\mathbb{Q},F) $ is strongly minimal.

Hence, $ \psi(x) = \exists y R(x,y) $ is a strongly minimal formula in $ \operatorname{Th}(\mathcal{M}') $. Suppose $ \mathcal{N}_{1} $ and $ \mathcal{N}_{2} $ are $ \aleph_{1} $-models of $ \operatorname{Th}(\mathcal{M}') $. In $ \mathcal{M} $, for every $ a$ and $b $ with $ R^{\mathcal{M}}(a,b) $, the restriction of $ f^{\mathcal{M}}(a,b,\cdot) $ to $ \psi(\mathcal{M}) $ is one-to-one and onto $ R(a,\mathcal{M}) $---in fact, it is an isomorphism of $ \{f\} $-structures. This is expressible in $ \mathcal{L}' $, and is therefore known to $ \operatorname{Th}(\mathcal{M}) $, and consequently each $ \mathcal{N}_{i} $ is a $ (|\psi(\mathcal{N}_{i})| +1) $-sized union of sets of size $ |\psi(\mathcal{N}_{i})| $, and so $ |\psi(\mathcal{N}_{i})| = \aleph_{1} $. Therefore, as the $ \psi(\mathcal{N}_{i}) $'s are models of the uncountably-categorical theory, $ \operatorname{Th}(\mathbb{Q},F) $, there is an $ \{ f \} $-isomorphism $ g: \psi(\mathcal{N}_{1}) \to \psi(\mathcal{N}_{2}) $.
Extend $ g $ to all of $ \mathcal{N}_{1} $ by choosing, for each $ a \in \psi(\mathcal{N}_{1}) $, a point $ c_{a} $ with $ R^{\mathcal{N}_{1}}(a,c_{a}) $, and likewise for each element of $ \psi(\mathcal{N}_{2}) $. Then if $ d \in \mathcal{N}_{1} \setminus \psi(\mathcal{N}_{1}) $, say if $ R^{\mathcal{N}_{1}}(a,d) $, let $ g(d) = f^{\mathcal{N}_{2}}(g(a),c_{g(a)},g((f^{\mathcal{N}_{1}})^{-1}(a,c_{a},d))) $, where $ (f^{\mathcal{N}_{1}})^{-1}(a,c_{a},d) $ denotes the (unique) element, $ b $, of $ \psi(\mathcal{N}_{1}) $, for which $ f^{\mathcal{N}_{1}}(a,c_{a},b) = d $. One then verifies that this extension of $ g $ is an $ \mathcal{L}' $-isomorphism, by again appealing to first-order properties of $ \mathcal{M} $ true in the $ \mathcal{N}_{i} $.
\end{proof}

So $ \operatorname{Th}(\mathcal{M}) $ is uncountably categorical, and one can easily see that $ \mathcal{M} $ is the prime model. 

\begin{prop} \label{not locally (OB)}
$ \operatorname{Aut}(\mathcal{M}) $ is not locally (OB). That is, the coarse structure associated to the relatively (OB) subsets of $ \operatorname{Aut}(\mathcal{M}) $ cannot be given by a metric.
\end{prop}

\begin{proof}
Observe that every element of the structure $ (\mathbb{Q},F) $ introduced in the proof of Proposition \ref{categorical} is definable over $ \{0,1\} $. For this, let $ F_{(a,b)} $ denote $ F(a,b,\cdot) $, and observe that for $ n \in \mathbb{N} $,
\[ n = F_{(0,1)}^{n}(0) \text{ and } -n = F_{(1,0)}^{n}(0), \]
while $ \frac{k}{n} \in \mathbb{Q} $ ($ n \geq 1 $) is the unique $ x $ for which,
\[ k = F_{(0,x)}^{n}(0). \]
Therefore, every automorphism of $ (\mathbb{Q},F) $ is determined by where it sends $ 0 $ and $ 1 $.

Now suppose $ g \in \operatorname{Aut}(\mathcal{M}) $. Then as every point in the strongly minimal set, $ \psi(\mathcal{M}) $, is definable over $ \emptyset $ (recall $ \mathcal{L} $ contains symbols for $ 0 $ and $ 1 $), $ g $ must fix $ \psi(\mathcal{M}) $ pointwise, and so for every $ a \in \psi(\mathcal{M}) $, fixes $ R(a,\mathcal{M}) $ setwise. Let $ g_{a}: \mathbb{Q} \to \mathbb{Q} $ be the automorphism of $ (\mathbb{Q},F) $ induced by the action of $ g $ on $ \psi(a,\mathcal{M}) $, satisfying $ g(a,b)=(a,g_{a}(b)) $.

Recall that $ f^{\mathcal{M}}(p,(p,q),r) = (p,(q-p)+r) $. So for any $ x \in \mathbb{Q} $,
\begin{align*}
(a,g_{a}(x)) &= g(a,x) \\
&= g(a,(0-a)+(x+a)) \\
&= g(f^{\mathcal{M}}(a,(a,0),x+a)) \\
&= f^{\mathcal{M}}(g(a),g(a,0),g(x+a)) \\
&= f^{\mathcal{M}}(a,(a,g_{a}(0)),x+a) \text{ (since } g \upharpoonright \psi(\mathcal{M}) = \operatorname{id} ) \\
&= (a,(g_{a}(0) - a ) + (x+a)) \\
&= (a,x + g_{a}(0))
\end{align*}
and $ g_{a} $ is a translation by $ g_{a}(0) $.

So every $ g \in \operatorname{Aut}(\mathcal{M}) $ fixes $ \psi(\mathcal{M}) $ and restricts to a translation on each fiber, $ R(a,\mathcal{M}) $. In this way, it can be naturally identified with a point in $ \mathbb{Q}^{\mathbb{Q}} $ given by $ (g_{a}(0))_{a \in \mathbb{Q}} $. Conversely, suppose $ h \in \mathbb{Q}^{\mathbb{Q}} $. Let $ \widehat{h}: M \to M $ fix $ \psi(\mathcal{M}) $ and send $ (p,q) \mapsto (p,q + h(p)) $. Then $ \widehat{h} $ is easily seen to respect $ R,0,1 $, and the first, second, and fourth parts of the definition of $ f $, while for the third,
\begin{align*}
\widehat{h}(f^{\mathcal{M}}(p,(p,q),r)) &= \widehat{h}(p,(q-p)+r) \\
&= (p,(q-p)+r+h(p)) \\
&= f^{\mathcal{M}}(p,(p,q+h(p)),r) \\
&= f^{\mathcal{M}}(\widehat{h}(p),\widehat{h}(p,q),\widehat{h}(r)),
\end{align*}
and $ \widehat{h} \in \operatorname{Aut}(\mathcal{M}) $. So $ \operatorname{Aut}(\mathcal{M}) $ can be identified with $ \mathbb{Q}^{\mathbb{Q}} $, and as a basic open set in $ \operatorname{Aut}(\mathcal{M}) $ is determined by its action on finitely-many points (i.e., fixes the values of $ h(a) $ for finitely-many $ a $), they are isomorphic as topological groups when the base $ \mathbb{Q} $ is given the discrete group topology.

Therefore, $ \operatorname{Aut}(\mathcal{M}) $ is an infinite product of groups that are not (OB), and so is not locally (OB), by Lemma \ref{products}.
\end{proof}

\nocite{MR1066691}

\bibliographystyle{amsalpha}
\bibliography{UncountablyCategoricalLSG}

\end{document}